\documentclass[psamsfonts]{amsart}
\usepackage{amssymb}
\usepackage{dsfont}
\usepackage[all]{xy}
\usepackage[mathcal]{eucal}
\usepackage{verbatim}
\usepackage{cite}
\DeclareMathOperator{\Hom}{Hom}

\DeclareMathOperator{\Ext}{Ext}
\DeclareMathOperator{\ad}{ad}
\DeclareMathOperator{\im}{im}

\newcommand{\Q}{\mathbb{Q}}

\newcommand{\C}{\mathbb{C}}

\newcommand{\TT}{\mathbb{T}}

\providecommand{\To}{\longrightarrow}

\providecommand{\cal}[1]{\mathcal{#1}}

\providecommand{\linepi}{\underline{\pi}}
\providecommand{\bb}[1]{\mathbb{#1}}
\newtheorem{thrm}{Theorem}[section]

\newtheorem{lemma}[thrm]{Lemma}

\newtheorem{conj}[thrm]{Conjecture}
\newtheorem*{thrm1}{Theorem \ref{GHforrat'lTTspectra}}

\theoremstyle{definition}
\newtheorem{defn}[thrm]{Definition}

\theoremstyle{remark}

\makeatletter
\let\c@equation\c@thrm
\makeatother
\numberwithin{equation}{section}

\title{The Equivariant Generating Hypothesis}
\author{Anna Marie Bohmann}

\begin{document}
\begin{abstract}We state the generating hypothesis in the homotopy category of $G$-spectra for a compact Lie group $G$ and prove that if $G$ is finite, then the generating hypothesis implies the strong generating hypothesis, just as in the non-equivariant case.  We also give an explicit counterexample to the generating hypothesis in the category of rational $S^1$-equivariant spectra.
\end{abstract}
\maketitle
\section{The generating hypothesis in an equivariant context}
The Freyd conjecture, also known as the Freyd generating hypothesis, is a long-standing conjecture in stable homotopy theory.  Let $\mathrm{Ho}\cal{S}$ be the homotopy category of spectra and let $\pi_*(S)\text{-mod}$ be the category of $\pi_*(S)$-modules, where $S$ is the sphere spectrum.
\begin{conj}[Generating hypothesis]\label{WGHspectra}
The restriction of the functor 
\[\pi_*(-)\!:\mathrm{Ho}\cal{S}\to \pi_*(S)\text{-mod}\]
 to the subcategory of finite spectra is faithful.  That is, if a map $f:X\to Y$ between finite spectra $X$ and $Y$ induces the zero map $f_*\!:\pi_*(X)\to \pi_*(Y)$, then $f$ is nullhomotopic.
\end{conj}
This conjecture was introduced by Freyd in 1965 \cite{Freyd1966} and remains open today.  Recent work has examined analogous statements in algebraic categories that share many properties of the homotopy category of spectra, such as the derived category of a ring \cite{HLP2007,Lockridge2007} or the stable module category of a finite group \cite{BCCM2007,CCM2008}.  Here we examine the analogous conjecture for the homotopy category of $G$-equivariant spectra for a compact Lie group $G$.

The appropriate generalization of the generating hypothesis to an equivariant context must take into account the homotopy groups of the fixed point spectra for each closed subgroup $H\leq G$.  This is best formulated in the language of Mackey functors. Let $\cal{M}[G]$ be the category of Mackey functors over a compact Lie group $G$, and let $\mathrm{Ho}G\cal{S}$ be the homotopy category of $G$-spectra.
\begin{conj}[Equivariant generating hypothesis]\label{equivGH}  The restriction of the equivariant homotopy functor $\linepi_*^G(-)\!: \mathrm{Ho}G\cal{S}\to \cal{M}[G]$ to the subcategory of finite $G$-spectra is faithful.  That is, if a map $f\!:X\to Y$ between finite $G$-spectra $X$ and $Y$ induces the zero map $f_*\!:\linepi^G_*(X)\to \linepi_*^G(Y)$, then $f$ is nullhomotopic.
\end{conj}
A map $M\to N$ of Mackey functors is zero if and only if the map $M(G/H)\to N(G/H)$ is zero for each closed subgroup $H\leq G$. Recall that the Mackey functor $\linepi_n^G(X)$ is given at $G/H$ by 
\[\linepi_n^G(X)(G/H)=[S^n\wedge G/H_+,X]=\pi_*^H(X).\]
Hence this conjecture can also be stated as requiring a map $f\!:X\to Y$ of finite $G$-spectra to be nullhomotopic if it induces the zero map $f_*\!:\pi_*^H(X)\to \pi_*^H(Y)$ for all closed $H\leq G$. Since the conjecture does not make use of the full structure of a Mackey functor, we will not give an explicit definition here but instead refer the reader to \cite{Dress1973}.

The above formulation of the  Freyd conjecture (\ref{WGHspectra}) is also known as the \emph{weak generating hypothesis}, as opposed to the \emph{strong generating hypothesis}.  Freyd's original formulation of the  strong generating hypothesis  conjectures that the functor $\pi_*(-)$ is both faithful and full when restricted to the subcategory of finite spectra.  In his initial paper on  the subject, Freyd also proves
\begin{thrm}[\hspace{-.5pt}{\cite[Prop.~9.7]{Freyd1966}}] If the weak generating
  hypothesis (Conj.~\ref{WGHspectra}) holds, then so does the strong
  generating hypothesis.  That is, if the restriction of the functor
  $\pi_*\!:\mathrm{Ho}\cal{S}\to \pi_*(S)\text{-mod}$ to the
  subcategory of finite spectra is faithful, it is
  also full on this subcategory.
\end{thrm}

This type of ``faithful implies full'' implication does not always hold, as shown by Hovey, Lockridge, and Puninski in the case of the derived category of a ring \cite{HLP2007}. However, this implication is true in the equivariant case for a finite group $G$.  

\begin{thrm}\label{equivfaithfulimpliesfull} Let $G$ be a finite group, and let $\mathcal{M}[G]$ be the
  category of Mackey functors over $G$. Then, if the functor
  $\linepi_*^G(-):\text{Ho}G\cal{S}\to \cal{M}[G]$ is faithful on
  restriction to the full subcategory of finite spectra, it is also full on this subcategory.
\end{thrm}

 We also consider the generating hypothesis in the category of rational $G$-spectra.  For a finite group $G$, Greenlees and May \cite[App. A]{GM1995} have proved that there is a natural equivalence between rational $G$-spectra and the category of graded Mackey functors over $\underline{A}\otimes \Q$, where $\underline{A}$ is the Burnside ring Mackey functor.  This equivalence is induced by the natural isomorphism 
\begin{align}\label{rat'lhomotopyisequivalence}\vartheta: [X_0,Y_0]^G\to \prod_n \Hom_{\mathcal{M}[G]}(\linepi_n^G(X_0),\linepi_n^G(Y_0))
\end{align}
for rational $G$-spectra $X_0$ and $Y_0$.  This isomorphism  completely algebraicizes the rational $G$-stable homotopy category, and thus in particular establishes that the strong generating hypothesis holds in this category. As an aside, we note that the isomorphism (\ref{rat'lhomotopyisequivalence}) is also used in proving Theorem \ref{equivfaithfulimpliesfull}.

For infinite compact Lie groups the picture is quite different. We  prove that the weak generating hypothesis fails rationally in the simplest case of a non-finite group of equivariance, that is, for the category of rational $\TT$--spectra, where $\TT$ is the compact Lie group $S^1$.

\begin{thrm}\label{GHforrat'lTTspectra}
The generalized Freyd conjecture (\ref{equivGH})  does not hold for the category of rational $\TT$-equivariant spectra.
\end{thrm}

This result suggests that the generating hypothesis should not hold for other infinite groups of equivariance, and it also makes immaterial the question of whether the weak generating hypothesis implies the strong generating hypothesis in the category of rational $\TT$-equivariant spectra. 

We prove Theorem \ref{equivfaithfulimpliesfull} in Section \ref{prooffaithfull}. As a technical tool for this proof, in Section \ref{abeliancat} we introduce the construction of an abelian envelope of a triangulated category.  In Section \ref{rational}, we establish Theorem \ref{GHforrat'lTTspectra} by giving a specific counterexample to the generating hypothesis in the category of rational $\TT$--spectra.  The structure of this proof is due in part to John Greenlees, and the author would like to thank him for his guidance and conversation.

\section{The abelian envelope of a triangulated category}\label{abeliancat}
Our proof of Theorem \ref{equivfaithfulimpliesfull} follows Freyd's proof that the weak generating hypothesis implies the strong generating hypothesis for the category of spectra \cite[Prop.~9.7]{Freyd1966}.  In particular, we make use of  the abelian category $\cal{A}(G)$ associated to the equivariant stable homotopy category. This is a special case of the general construction of an abelian envelope of a triangulated category.  This construction was first given in \cite{Freyd1966}; for a more modern treatment, including several equivalent constructions, see chapter 5 of Neeman's book on triangulated categories \cite{Neeman2001}.

Given a triangulated category $\cal{T}$, its abelian envelope $\cal{A(T)}$ is an abelian category with a full and faithful inclusion $\iota\!: \cal{T}\to \cal{A(T)}$ with the properties that
\begin{equation}\label{propertiesofA}\hfill
\begin{minipage}{.91\textwidth}
\begin{itemize}
\item for any object $s$ in $\cal{T}$, its image $\iota(s)$ in $\cal{A(T)}$ is projective and injective, and every projective or injective object in $\cal{A(T)}$ is a direct summand of an object in the image of $\iota$,
\item $\cal{A(T)}$ is self-dual, and
\item any homological functor $H\!:\cal{T}\to \cal{A}$, where $\cal{A}$ is an abelian category, extends uniquely to an additive functor $\cal{A(T)}\to \cal{A}$.
\end{itemize}
\end{minipage}
\end{equation}
By homological functor, we mean a functor that takes exact triangles of $\cal{T}$ to exact sequences in $\cal{A}$. Note that self-duality implies we get a similar extension of any cohomological functor on $\cal{T}$ as well. 

We define $\cal{A(T)}$ to be the following quotient of the category of maps in $\cal{T}$; this is Freyd's original description \cite[\S 3]{Freyd1966} and is also given by Neeman \cite[5.2.1]{Neeman2001}.
\begin{defn}\label{defnabeliancat} The objects of $\cal{A(T)}$ are maps $s\to t$ in $\cal{T}$, and  morphisms in $\cal{A(T)}$ are commutative diagrams
\[
\xymatrix{s\ar[r]\ar[d]&t\ar[d]\\
s'\ar[r]&t'}
\]
under the additive equivalence relation defined by setting such a morphism equal to zero if the composite $s\to t\to t'$ or $s\to s'\to t'$ is zero.  The embedding $\iota\!:\cal{T}\to \cal{A}(\cal{T})$ is given by sending an object $s$ to the identity morphism $s\to s$.
\end{defn}

We can think of the functor $\iota$ as a universal homological functor from $\cal{T}$ to an abelian category.  Because $\iota$ is full and faithful, we will usually identify an object $s\in\cal{T}$ with its image under $\iota$. For a proof that $\cal{A(T)}$ has the desired properties (\ref{propertiesofA}), see \cite[Ch.~5]{Neeman2001}.

In order to make use of this construction in proving Theorem \ref{equivfaithfulimpliesfull}, we need to show that if a version of the generating hypothesis holds in a triangulated category $\cal{T}$, it also holds in the category $\cal{A(T)}$.

\begin{lemma} \label{extendGH}Let $\cal{T}$ be a triangulated category, let $\cal{A}$ be an abelian category,  and let $H\!:\cal{T}\to \cal{A}$ be a homological functor. If $H$ is faithful on restriction to the category of compact objects of $\cal{T}$, then its extension to $\cal{A(T)}$ is also faithful on restriction to objects with a projective resolution by compact objects of $\cal{T}$.
\end{lemma}
\begin{proof}  Objects in $\cal{A(T)}$ with a projective resolution by
  compact objects of $\cal{T}$ are the same as objects $s\to t$ where
  $s$ and $t$ are compact objects of $\cal{T}$ \cite[\S
  5]{Neeman2001}. The value of a homological functor $H$ on an object
  $s\to t$ in $\cal{A}(\cal{T})$ is given by the image of the map $H(s)\to H(t)$. If $s$, $t$, $s'$ and $t'$ are compact objects of $\cal{T}$ and we have a morphism 
\[
\xymatrix{s\ar[r]^{i}\ar[d]&t\ar[d]\\
s'\ar[r]^{i'}&t'}
\]
in $\cal{A(T)}$,then by definition the map $H(s\to t)\to H(s'\to t')$
is zero if and only if the induced map $\im H(i)\to \im H(i')$ is zero.  But if this is the case, then the map $H(s)\to H(s')\to H(t')$ is zero.  Hence, assuming $H$ is faithful on compact objects of $\cal{T}$, the map $s\to s'\to t'$ is zero, and so our original morphism
%\[
%\xymatrix{ s\ar[r]\ar[d]&t\ar[d]\\
%s'\ar[r]&t'}
%\]
is also zero.
\end{proof}
This lemma says that in any triangulated category where a version of the generating hypothesis holds, the generating hypothesis extends to the abelian envelope of the category.  We will make use of this result in the case $\cal{T}=\text{Ho}G\cal{S}$ in the next section.

\section{The proof of Theorem \ref{equivfaithfulimpliesfull}: Faithful implies full}\label{prooffaithfull}

We now prove that for a finite group $G$, the generating hypothesis for $G$-spectra implies the strong generating hypothesis for $G$-spectra.  Let $\cal{A}(G)$ be the abelian envelope of the triangulated category $\text{Ho}G\cal{S}$ as discussed in Section \ref{abeliancat}.  In particular,  $\cal{A}(G)$ is an abelian category with a full embedding $\text{Ho}G\cal{S}\to\cal{A}(G)$ such that the image of every $G$-spectrum is both projective and injective.  Since the embedding is full and faithful, we identify the group of maps $\cal{A}(G)(X,Y)$ with $[X,Y]^G$ for $G$-spectra $X$ and $Y$.

\begin{proof}[Proof of Theorem \ref{equivfaithfulimpliesfull}]
Let $G$ be a finite group. Let $X$ and $Y$ be finite $G$-CW spectra,
and let $S$ be the sphere $G$-spectrum.  Let $\linepi_*^G(S)$ be the graded Green functor that is $\linepi_n^G(S)$ in degree $n$, meaning that $\linepi_*^G(S)$  is a graded Mackey functor with a compatible ring structure.   Note that $\linepi_0^G(S)$ is the Burnside ring Green functor.   Similarly, let $\linepi_*^G(X)$ be the graded Mackey functor that is $\linepi_n^G(X)$ in degree $n$. Composition makes $\linepi_*^G(X)$ a module over $\linepi_*^G(S)$.

Suppose we have a natural transformation $\phi:\linepi_*^G(X)\to
\linepi_*^G(Y)$. To prove that $\linepi_*^G(-)$ is
full, we must find a map $f\!:X\to Y$ such that
$\linepi_*^G(f)=\phi$.  The structure of the proof
is in two steps.  First, we use the embedding
$\iota\!:\mathrm{Ho}G\cal{S}\to \cal{A}(G)$ to construct a sequence of
maps $\{f_n\!:X\to Y\}$ such that  $(f_{n})_*\!:\linepi_j^G(X)\to
\linepi_j^G(Y)$ agrees with $\phi$ for $j\leq n$.  We then show
that $\{f_n\}$ has a constant cofinal subsequence whose value is the
desired map $f$.

For each conjugacy class of
 subgroups $H\leq G$ and each $j$, choose a set of generators of $\linepi_j^G(X)(G/H)=[\Sigma^j(S\wedge G/H_+),X]^G$.  Denote these generators by $\{\alpha_{j,H,i}\!:\Sigma^j(S\wedge G/H_+)\to X\}$. Let 
\[W_n=\bigvee_{j\leq n}\bigvee_{H\leq G}\bigvee_i \Sigma^j(S\wedge G/H_+).\]
Since $X$ is a finite spectrum, its homotopy groups are finitely generated and thus $W_n$ is a finite spectrum as well. Let $\alpha_n\!:W_n\to X$ be the map given on wedge summands by $\alpha_{H,j,i}$.  Similarly,  let $h_n\!:W_n\to Y$ be the map given on wedge summands by $\phi(\alpha_{H,j,i})$.

Let $\iota_n\!:K_n\to W_n$ be the kernel of the map $\alpha_n\!:W_n\to X$ in  $\cal{A}(G)$. Thus, the sequence $0\to K_n\to W_n\to X$ is exact in $\cal{A}(G)$.  We first show that
$\linepi_*^G(h_n\circ \iota_n)=0$.  Then we apply the
assumption that $\linepi_*^G(-)$ is faithful to conclude that $h_n\circ\iota_n$ is zero.  By construction,
and because $\phi$ is
a natural transformation of Mackey functors,
\begin{align*}
\linepi_*^G(h_n\circ\iota_n)&=\linepi_*^G(h_n)\circ\linepi_*^G(\iota_n)\\
&=\linepi_*^G(\phi(\alpha_n))\circ\linepi_*^G(\iota_n)\\
&=\phi(\linepi_*^G(\alpha_n))\circ\linepi_*^G(\iota_n)\\
&=\phi(\linepi_*^G(\alpha_n\circ\iota_n)).
\end{align*}
 Since $\alpha_n\circ \iota_n =0$ in $\cal{A}(G)$, we see that
$\linepi_*^G(h_n\circ
\iota_n)=\phi(\linepi_*^G(\alpha_n\circ \iota_n))=0$. 
We are assuming the generating hypothesis holds in
$\textrm{Ho}G\cal{S}$, so we can apply Lemma \ref{extendGH} to conclude that
$h_n\circ \iota_n:K_n\to W_n\to Y$ is zero in $\cal{A}(G)$. 

Now we construct our maps $\{f_n\}$. Let $\zeta_n\!:W_n\to D_n$ be the cokernel of the map $\iota_n\!:K_n\to
W_n$. Note that, as $K_n$ is the kernel of $\alpha_n\!:W_n\to X$, the
map $\alpha_n\!:W_n\to X$ factors through an injection
$\nu_n\!:D_n\to X$.  Since $h_n\circ\iota_n=0$, the map
$h_n$ factors through a map $\gamma_n\!:D_n\to Y$.  Also let
$\mu_n\!:X\to C_n$ be the cokernel of $\alpha_n\!:W_n\to X$.  The object
$Y$ is injective in the category $\cal{A}(G)$, as mentioned in (\ref{propertiesofA}), so the map $\gamma_n:\!D_n\to Y$ extends to a map $f_n\!:X\to Y$ such that $f_n\circ \nu_n=\gamma_n$.  We
summarize these definitions in the following commutative diagram in
$\cal{A}(G)$.
\[
\xymatrix{0\ar[r] & K_n\ar[r]^-{\iota_n}\ar[rdd]_-{0} &W_n\ar[rr]^-{\alpha_n}\ar[dr]^-{\zeta_n}\ar[dd]_-{h_n} &&X\ar[r]^-{\mu_n}\ar@/^6ex/[ddll]^{f_n} &C_n\ar[r] & 0\\
&&& D_n\ar[ur]^-{\nu_n}\ar[dl]_-{\gamma_n}\\
&&Y}
\]
We have constructed $\alpha_n$ such that $\linepi_*^G(\alpha_n)$
is a surjection $\linepi_j^G(W_n)\to\linepi_j^G(X)$ for $j\leq n$.
Since
$\linepi_*^G(f_n)\linepi_*^G(\alpha_n)=\linepi_*^G(h_n)=\phi\left(\linepi_*^G(\alpha_n)\right)$,
it follows that $\linepi_j^G(f_n)=\phi$ for $j\leq n$.  We need to
construct a map $f=f_\infty$ that has this property for all $n$.

By construction, we
have an inclusion $i_{n}\!:W_n\to W_{n+1}$ such that
$\alpha_n=\alpha_{n+1}\circ i_{n}$ and $h_n=h_{n+1}\circ i_{n}$.
For any $m>n$, iterating these inclusions of summands gives an inclusion
$W_n\to W_m$  which makes the diagram  
\[
\xymatrix{W_n\ar[r]\ar[d]_-{h_n}\ar@/^3ex/[rr]^{\alpha_n}& W_m\ar[dl]^-{h_m}\ar[r]_-{\alpha_m}&X\ar@/^3ex/[dll]^-{f_m}\\
Y}
\]
commute.
Hence the restriction of the map $f_m\circ\alpha_m\!:W_m\to X\to Y$ to
the summands contained in $W_n$ is the map $h_n$; that is, $f_m\circ\alpha_n=h_n$. By the universal
property of the cokernel $\mu_n\!:X\to C_n$, we obtain a factorization
of the map $f_m-f_n$ through $C_n$.
\[
\xymatrix{W_n\ar[r]^-{\alpha_n}\ar@/_2ex/[dr]_-{0}&X\ar[r]^{\mu_n}\ar[d]_-{f_m-f_n\!\!}&C_n\ar@{.>}[dl]\\
&Y}
\]

Since $X$ is finite, we may choose $d$ large enough so that
$H_j^G(X;\underline{A})=0$ for all $j>d$, where $\underline{A}$ is the
Burnside ring Mackey functor for the group $G$ and $H^G_*$ is Bredon homology. We claim that $\cal{A}(G)(C_d,Y)$ is finite, or equivalently, since $\cal{A}(G)(C_d,Y)$ is finitely generated, that the group $\cal{A}(G)(C_d,Y)\otimes \Q=0$. From this it follows that we may extract the desired  constant
cofinal subsequence whose common value $f$ satisfies
$\linepi_*^G(f)=\phi$.

 First consider $\underline{H}_*^G(C_d;\underline{A})$.  Since the functor
 $\underline{H}_j^G(-;\underline{A})\!:\cal{A}(G)\to \cal{M}[G]$ is exact, for each $H\leq G$ we have an exact sequence of abelian groups
\begin{align}\label{homologyexactseq} \underline{H}_j^G(W_d;\underline{A})(G/H)\stackrel{(\alpha_d)_*}{\To}
 \underline{H}_j^G(X;\underline{A})(G/H)\to
 \underline{H}_j^G(C_d;\underline{A})(G/H)\to 0
\end{align}
for each $j$.  When $j>d$, the Mackey functor $\underline{H}_j^G(X;\underline{A})$ is zero, and
thus $\underline{H}_j^G(C_d;\underline{A})$ is zero as well.  As May
and Greenlees show \cite[App.~A]{GM1995}, the rationalization of the sphere $G$-spectrum is
the Eilenberg--Maclane $G$-spectrum $H(\underline{A}\otimes\Q)$.  Thus
we can equate the rationalization of the sequence
(\ref{homologyexactseq}) with the sequence
\[ \linepi_j(W_d)\otimes\Q\stackrel{(\alpha_{d})_*}{\To} \linepi_j(X)\otimes\Q\to
\linepi_j(C_d)\otimes \Q\to 0.\]
When $j\leq d$, the map $\alpha_d$ is surjective on rational homotopy
by construction, so $\linepi_j(C_d)\otimes \Q=0$ for $j\leq
d$. Since $\linepi_j(C_d)\otimes\Q=\underline{H}_j^G(C_d;\underline{A})\otimes\Q=0$ for $j>d$, the graded Mackey functor $\linepi_*^G(C_d)\otimes\Q$ is zero.

Denote by $Z_0$ the rationalization of a $G$-spectrum $Z$; then
\[[Z,Z']_*^G\otimes \Q=[Z_0,Z'_0]_*^G=[Z,Z'_0]_*^G.\]  By restricting to the category of rational $G$-spectra, we in
principal get an abelian category $\cal{A}_0(G)$ for rational stable
homotopy as in the construction of Definition
\ref{defnabeliancat}. However, the rationalization functor
$\text{Ho}G\cal{S}\to \text{Ho}G\cal{S}_0$ induces an extension
$\cal{A}(G)\to \cal{A}_0(G)$.  This functor is an equivalence after
tensoring with $\Q$; this does not change the category $\cal{A}_0(G)$
since $\cal{A}_0(G)$ is already rational.  Hence we do not distinguish
between $\cal{A}(G)\otimes\Q$ and $\cal{A}_0(G)$.

For $G$-spectra $Z$ and $Z'$, we restate the isomorphism (\ref{rat'lhomotopyisequivalence}) in terms of
 tensoring with $\Q$:
\begin{equation}\label{Qtensorratlhtpyequiv}\vartheta: [Z,Z']^G\otimes\Q\to \prod_n
 \Hom_{\cal{M}[G]}(\linepi_n^G(Z)\otimes\Q,\linepi_n^G(Z')\otimes\Q).
\end{equation}
 Moreover, for a fixed $Z'$, both sides of (\ref{Qtensorratlhtpyequiv})  are cohomology theories on
$G$-spectra, so this isomorphism extends to $\cal{A}(G)$. Setting $Z'=Y$ then allows us to conclude that the group $\cal{A}(G)(C_d,Y)\otimes\Q$ is zero. More precisely, the columns are
 exact in the following diagram.
\[
\xymatrix{[W_d,Y]^G\otimes \Q \ar[r]\ar[d]&
  \prod\limits_n\Hom_{\mathcal{M}[G]}(\linepi_n^G(W_d)\otimes \Q,\linepi_n^G(Y)\otimes \Q)\ar[d]\\
[X,Y]^G\otimes \Q\ar[r]\ar[d]& \prod\limits_n
  \Hom_{\mathcal{M}[G]}(\linepi_n^G(X)\otimes \Q,\linepi_n^G(Y)\otimes\Q)\ar[d]\\
\cal{A}(G)(C_d,Y)\otimes\Q\ar[r]\ar[d]\ar[r]\ar[d]&\prod\limits_n\Hom_{\mathcal{M}[G]}(\linepi_n^G(C_d)\otimes\Q,\linepi_n^G(Y)\otimes
  \Q)\ar[d]\\
 0\ar[r]&0}
\]
Here we are identifying $\cal{A}(G)(Z,Y)$ with $[Z,Y]^G$ for
a $G$-spectrum $Z$.   By (\ref{Qtensorratlhtpyequiv}), the top two rows are isomorphisms and hence the five lemma implies that 
\[\cal{A}(G)(C_d,Y)\otimes\Q\cong \prod
\Hom_{\mathcal{M}[G]}(\linepi_n^G(C_d)\otimes\Q,\linepi_n^G(Y)\otimes
\Q).\]  
But $\linepi_*^G(C_d)\otimes\Q=0$, so in fact
$\cal{A}(G)(C_d,Y)\otimes \Q=0$. 

Therefore, the sequence $\{f_m\}$ has a constant cofinal subsequence, and we can take its value to be the desired map $f\!:X\to Y$ for which $\linepi_*^G(\overline{f})=\phi$.
\end{proof}

Thus, for a finite group $G$, the weak generating hypothesis implies
the strong generating hypothesis, just as in the non-equivariant
case.  

\section{The rational equivariant case}\label{rational}
Restricting to the category of rational $G$-spectra for a group $G$ greatly simplifies the structure of the category, and after this simplification the rational $G$-stable homotopy categories for finite and infinite compact Lie groups exhibit very different behavior.  As mentioned in the introduction, the strong generating hypothesis holds in the rational stable homotopy category for a finite group $G$, but even the weak generating hypothesis fails in the category of $\TT$-spectra, where $\TT$ is the circle group. In fact, using Greenlees's algebraic model for the category of rational $\TT$-spectra, we find an explicit counterexample to the generating hypothesis in this category.

\begin{thrm1}
The generalized Freyd conjecture does not hold for the category of rational $\TT$--equivariant spectra.
\end{thrm1}

Before proving this theorem, we prove the following lemma about the equivariant homotopy groups of the suspension spectrum $\bb{X}=\Sigma_+^\infty X$ of a free $\TT$-space $X$.

\begin{lemma} \label{ratlhomotopygrps}Let $X$ be a free rational $\TT$-space, and let $\bb{X}=\Sigma_+^\infty X$ denote the suspension spectrum of $X_+$.  Then, for any closed subgroup $H\leq \TT$,
\[\pi_*^H(\bb{X})=
\begin{cases}\pi_*(\Sigma\bb{X}/\TT) &\text{ if }H=\TT\\
\pi_*(\bb{X}) & \text{ if } H<\TT.
\end{cases}
\]
\end{lemma}
\begin{proof}
The
tom Dieck splitting theorem ~\cite{LMS1986} states that for such an $\bb{X}$ and for any closed subgroup $H\leq \TT$, 
\begin{equation}\label{tDsplit}
\pi_*^H(\bb{X})=\sum_{K\leq H}\pi_*\left(\Sigma^\infty(EW_HK_+\wedge_{W_H K}\Sigma^{\ad(W_HK)} X^K_+)\right),
\end{equation}
where the sum runs over conjugacy classes of closed subgroups $K\leq H$ and $\ad(W_HK)$ is the adjoint representation of the Weyl group $W_HK$.  Let us first consider $\pi_*^\TT(\bb{X})$. Since $X$ is free, the space $X^K_+$ is a point unless $K$ is the trivial group. Hence the only summand that contributes to $\pi_*^\TT(\bb{X})$ is the summand corresponding to the trivial group $1\leq \TT$. The Weyl group of $1\leq \TT$ is $\TT$ itself, and the adjoint representation of $\TT$ is trivial. Thus
\begin{align*}\pi_*^\TT(\bb{X})&=\pi_*\left(\Sigma^\infty(E\TT_+\wedge_\TT\Sigma X_+)\right)\\
&=\pi_*(\Sigma \bb{X}/\TT).
\end{align*}

Now suppose $H$ is a proper closed subgroup of $\TT$.  Then $H$ is finite, so $\ad(W_HK)$ is zero for all $K\leq H$. Since $X$ is free, the summand of (\ref{tDsplit}) corresponding the trivial group is the only summand contributing to $\pi^H_*(\bb{X})$, just as in the previous case. 
Hence $\pi_*^H(\bb{X})$ reduces to 
\begin{align*}\pi_*^H(\bb{X})&=\pi_*\left(\Sigma^\infty(EH_+\wedge_H X_+)\right)\\
&=\pi_*(\bb{X}/H).
\end{align*} 
 Since we are working
rationally, $\pi_*(\bb{X}/H)=H_*(\bb{X}/H)$.  As $H$ is necessarily finite, $H_*(\bb{X}/H)=H_*(\bb{X})/H$, but
since $H$ is acting as a subgroup of the connected group $\TT$, $H$
must in fact act trivially on $H_*(\bb{X})$.  Thus $\pi_*(\bb{X}/H)=H_*(\bb{X})=\pi_*(\bb{X})$, where the second equality again follows by rationality.  That is,
for any proper closed subgroup $H\leq \TT$, the  $H$-homotopy groups
$\pi_*^H(\bb{X})$ are in fact the non-equivariant homotopy groups
$\pi_*(\bb{X})$.
\end{proof}

\begin{proof}[Proof of Theorem \ref{GHforrat'lTTspectra}] We restrict our attention to free rational $\TT$-spectra and use the results of Greenlees in ~\cite{Greenlees1999}.  Throughout the following proof, we assume that all spectra are rational without explicitly indicating this in our notation.

Let $\bb{X}$ and $\bb{Y}$ be finite rational $\TT$-spectra.  The equivariant Freyd conjecture asserts that a map $f:\bb{X}\to \bb{Y}$ is zero if and only if it induces the zero map on $\pi_*^H(-)$ for all closed subgroups $H\leq \TT$. We give a counterexample where $\bb{X}$ and $\bb{Y}$ happen to be free rational $\TT$-spectra.

Greenlees proves ~{\cite[3.1.1]{Greenlees1999}} that for free rational $\TT$-spectra, there is a natural Adams short exact sequence
\[0\to\Ext_{\Q[c]}(\pi_*^\TT(\Sigma \bb{X}),\pi_*^\TT(\bb{Y}))\to [\bb{X},\bb{Y}]^\TT_*\to \Hom_{\Q[c]}(\pi_*^\TT(\bb{X}),\pi_*^{\TT}(\bb{Y}))\to 0,\]
where $\Q[c]$ is the graded polynomial algebra an element $c$
in dimension $-2$.  This element $c$ comes from the Euler class of the representation of $\TT$ on $\C$ by multiplication \cite[Thm. 2.4.1]{Greenlees1999}.  Note that $\Hom$ and $\Ext$ are graded, and $\Ext$ here means $\Ext^{1,*}$.  This
short exact sequence is our main tool in finding a counterexample to the
equivariant Freyd conjecture.

Let $X$ and $Y$ be free rational $\TT$-spaces, and let $\bb{X}=\Sigma_+^\infty X$ and $\bb{Y}=\Sigma^\infty_+ X$.  
Consider the space
$\TT_+\wedge X_+$ with $\TT$ acting diagonally.  There is an isomorphism $[\Sigma^\infty(\TT_+\wedge X_+),
\bb{Y}]^\TT_*\cong [\bb{X},\bb{Y}]_*$ between equivariant homotopy classes of maps from
$\Sigma^{\infty}(\TT_+\wedge X_+)$ to $\bb{Y}$ and non-equivariant homotopy classes of maps
from $\bb{X}$ to $\bb{Y}$, and since $\bb{X}$ and $\bb{Y}$ were assumed to be rational,
$[\bb{X},\bb{Y}]_*$ is also isomorphic to $\Hom_\Q(\pi_*(\bb{X}),\pi_*(\bb{Y}))$. 

Write $\Q[c]=k$ and $\Sigma^\infty(\TT_+\wedge X_+)=\bb{W}$ for conciseness.  The projection $p\!:\bb{W}\to \bb{X}$ induces a map of short exact
sequences
{%\small
\[
\xymatrix{0\ar[r]&\Ext_{k}(\pi_*^\TT(\Sigma
\bb{X}),\pi_*^\TT(\bb{Y}))\ar[r]\ar[d]_{p_*}&[\bb{X},\bb{Y}]^\TT_*\ar[d]_{p_*}\ar[r]&\Hom_{k}(\pi_*^\TT(\bb{X}),\pi_*^\TT(\bb{Y}))\ar[r]\ar[d]_{p_*}&0\\
0\ar[r]&\Ext_{k}(\pi_*^\TT(\Sigma \bb{W}),\pi_*^\TT(\bb{Y}))\ar[r]&[\bb{W},
 \bb{Y}]^\TT_*\ar[r]&\Hom_{k}(\pi_*^\TT(\bb{W}),\pi_*^\TT(\bb{Y}))\ar[r]&0}
\]}
which can be rewritten as 
\[
\xymatrix@C=9.5pt{0\ar[r]&\Ext_{k}(\pi_*^\TT(\Sigma
  \bb{X}),\pi_*^\TT(\bb{Y}))\ar[r]\ar[d]_{p_*}&[\bb{X},\bb{Y}]^\TT_*\ar[d]_{p_*}\ar[r]&\Hom_{k}(\pi_*^\TT(\bb{X}),\pi_*^\TT(\bb{Y}))\ar[r]\ar[d]_{p_*}&0\\
0\ar[r]&\Ext_{k}(\pi_*^\TT(\Sigma \bb{W}),\pi_*^\TT(\bb{Y}))\ar[r]&\Hom_\Q(\pi_*(\bb{X}),\pi_*(\bb{Y}))\ar[r]&\Hom_{k}(\pi_*^\TT(\bb{W}),\pi_*^\TT(\bb{Y}))\ar[r]&0}
\]
Let $f\!:\bb{X}\to \bb{Y}$ be a map inducing the zero
map $\pi_*^H(\bb{X})\to \pi_*^H(\bb{Y})$ for all subgroups $H\leq \TT$. By Lemma \ref{ratlhomotopygrps} this is equivalent to assuming that  $f_*:\pi_*^\TT(\bb{X})\to \pi_*^\TT(\bb{Y})$ and
$f_*:\pi_*(\bb{X})\to \pi_*(\bb{Y})$ are both zero.  In other words, $p_*(f)$ is zero in
\[\Hom_\Q(\pi_*(\bb{X}),\pi_*(\bb{Y}))=\Hom_\Q(\pi_*^H(\bb{X}),\pi_*^H(\bb{Y}))\] and also $f$ maps to zero in $\Hom_{k}(\pi_*^\TT(\bb{X}),\pi_*^\TT(\bb{Y}))$.  Hence $f$ is either zero, or lifts to a nontrivial element of
$\Ext_{k}(\pi_*^\TT(\Sigma \bb{X}),\pi_*^\TT(\bb{Y}))$.  To show that the
generating hypothesis fails for the category of rational
$\TT$--spectra, it is enough to show that the map
\begin{align}\label{kernel}
\Ext_{k}(\pi_*^\TT(\Sigma \bb{X}),\pi_*^\TT(\bb{Y}))\to
\Ext_{k}(\pi_*^\TT(\Sigma \bb{W}),\pi_*^\TT(\bb{Y}))
\end{align}
 has a
nontrivial kernel.  Any element of this kernel must be nonzero as an element of $[\bb{X},\bb{Y}]^\TT_*$, but must map to zero in the group $[\bb{W},\bb{Y}]^\TT_*$ by definition, and will map to zero in $\Hom_{k}(\pi_*^\TT(\bb{X}),\pi_*^\TT(\bb{Y}))$ by exactness. Thus such an element gives a counterexample to the generating hypothesis.

We now give a concrete example where the map (\ref{kernel}) has a nontrivial kernel. The group $\TT$ acts on the complex plane $\C$ by multiplication.  Consider the diagonal action of $\TT$ on $\C^a$, for some positive integer $a$. Let $S(a)$ be the unit sphere in this representation of $\TT$ and consider the rational suspension spectrum of $S(a)_+$.  We will denote this spectrum by $\bb{S}(a)$. We claim that for integers $a$, $b>1$, the spectra $\bb{S}(a)$ and $\bb{S}(b)$ provide a counterexample to the equivariant generating hypothesis.

To prove this claim, it suffices to calculate the groups 
\[\Ext_{\Q[c]}\left(\pi_*^\TT(\Sigma \bb{S}(a)),\pi_*^\TT(\bb{S}(b))\right) \textrm{ and } \Ext_{\Q[c]}\left(\pi_*^\TT(\Sigma \bb{W}),\pi_*^\TT(\bb{S}(b))\right)\] 
where now $\bb{W}=\Sigma^\infty(\TT_+\wedge S(a)_+)$. 
First note that Lemma \ref{ratlhomotopygrps} implies
$\pi_*^\TT(\bb{S}(a))=\pi_*(\Sigma \bb{S}(a)/\TT)$.
Since $\bb{S}(a)$ is the suspension spectrum of $S(a)_+$, we pass to the space level to see that
\begin{align*}
\pi_*(\Sigma\bb{S}(a))&=\pi_*(\Sigma \C P^{a-1})\\
&=\Sigma^{2a-1}\Q[c]/(c^a)
\end{align*}
as a graded module over $\Q[c]$. Similarly, Lemma \ref{ratlhomotopygrps} implies $\pi_*^\TT(\bb{W})=\pi_*(\Sigma \bb{S}(a))$, and passing to the space level yields
\begin{align*}\pi_*(\Sigma \bb{S}(a))&=\pi_*(\Sigma S^{2a-1})\\
&=\Sigma \Q\oplus \Sigma^{2a}\Q,
\end{align*}
again as a graded module over $\Q[c]$.  Hence our $\Ext$ calculations reduce to finding the groups
\[\Ext_{\Q[c]}\left(\Sigma^{2a}\Q[c]/(c^a),\Sigma^{2b-1}\Q[c]/(c^b)\right)\]  and  \[\Ext_{\Q[c]}\left(\Sigma^2\Q\oplus\Sigma^{2a+1}\Q,\Sigma^{2b-1}\Q[c]/(c^b)\right).\]  

To calculate the first group, we use the projective resolution
\[0\to \Q[c]\stackrel{c^{a}}{\To} \Sigma^{2a}\Q[c]\to \Sigma^{2a}\Q[c]/(c^a)\to 0
\]
of $\Sigma^{2a}\Q[c]/(c^a)$. 
Applying $\Hom_{\Q[c]}\left(-,\Sigma^{2b-1}\Q[c]/(c^b)\right)$ to this resolution gives an exact sequence
\begin{multline*}\dotsb\to\Sigma^{2b-1-2a}\Q[c]/(c^b)\stackrel{c^a}{\To} \Sigma^{2b-1}\Q[c]/(c^b) \\
\to \Ext_{\Q[c]}\left(\Sigma^{2a}\Q[c]/(c^a),\Sigma^{2b-1}\Q[c]/(c^b)\right)\to 0
\end{multline*}
so that 
\[\Ext_{\Q[c]}\left(\Sigma^{2a}\Q[c]/(c^a),\Sigma^{2b-1}\Q[c]/(c^b)\right)=
\begin{cases} \Sigma^{2b-1}\Q[c]/(c^a) & \text{ if $a\leq b$},\\ 
\Sigma^{2b-1}\Q[c]/(c^b) &\text{ if $a>b$.}
\end{cases}
\]

To calculate $\Ext_{\Q[c]}\left(\pi_*^\TT(\Sigma \bb{W}),\pi_*^\TT(\bb{S}(b))\right)$, we use the projective resolution
\[ 0\to \Q[c]\oplus\Sigma^{2a-1}\Q[c] \to \Sigma^2\Q[c]\oplus \Sigma^{2a+1}\Q[c]\to \Sigma^2\Q\oplus\Sigma^{2a+1}\Q\to 0\]
of $\Sigma^2\Q\oplus \Sigma^{2a+1}\Q$.  Applying $\Hom\left(-,\Sigma^{2b-1}\Q[c]/(c^b)\right)$ to this sequence gives the exact sequence
\begin{multline*}\dotsb\to \Sigma^{2b-3}\Q[c]/(c^b)\oplus\Sigma^{2b-2a-2}\Q[c]/(c^b)\\\to \Sigma^{2b-1}\Q[c]/(c^b)\oplus\Sigma^{2b-2a}\Q[c]/(c^b)\\\to \Ext_{\Q[c]}(\Sigma^2\Q\oplus\Sigma^{2a+1}\Q,\Sigma^{2b-1}\Q[c]/(c^b))\to 0.
\end{multline*}
Hence 
\[\Ext_{\Q[c]}\left(\Sigma^2\Q\oplus\Sigma^{2a+1}\Q,\Sigma^{2b-1}\Q[c]/(c^b)\right)=\Sigma^{2b-1}\Q\oplus \Sigma^{2b-2a}\Q.\]
  For any positive integer $n$, there are two non-zero maps from
  $\Sigma^{2b-1}\Q[c]/(c^n)$ to
  $\Sigma^{2b-1}\Q\oplus\Sigma^{2b-2a}\Q$, one of even degree and one
  of odd degree. If $n\geq 2$, neither of these maps is injective.
  Hence, for $a, b>1$, the map (\ref{kernel}) cannot be injective. Any
  nontrivial element of the kernel of (\ref{kernel}) yields a
  nontrivial map $\bb{S}(a)\to \bb{S}(b)$ that induces the trivial map
  on homotopy.
  Therefore $\bb{S}(a)$ and $\bb{S}(b)$ provide the desired counterexample to the equivariant generating hypothesis for the category of rational $\TT$-spectra.
\end{proof}

This result suggests that, at least rationally, the generating hypothesis should fail for any infinite compact Lie group, although we have not shown this.  The failure of the generating hypothesis in this case is perhaps surprising in light of the characterization of rings $R$ for which the strong generating hypothesis holds in the derived category $\cal{D}(R)$.  Hovey, Lockridge and Puninski  prove that the generating hypothesis holds in $\cal{D}(R)$ if and only if $R$ is von Neumann regular \cite[Thm.~1.3]{HLP2007}.  Since the rationalized Burnside ring of a compact Lie group is von Neumann regular, one might expect the generating hypothesis to hold in the rational equivariant stable category; however, we have shown that this is not the case.

\bibliography{GHbib}{}
\bibliographystyle{amsplain}
\end{document}